\documentclass{elsarticle}

\journal{---}

 \usepackage{amsmath,amssymb,graphicx,here,verbatim}
 \usepackage{latexsym, amscd}
 \usepackage{amsthm} 
 \usepackage{color}

\usepackage[T1]{fontenc}
\usepackage{mathrsfs}
\usepackage{enumitem}
\usepackage{stmaryrd}
\usepackage[left=3.4cm, right=3.4cm]{geometry}
\usepackage{hyperref}
\usepackage{url}

\newtheorem{theorem}{Theorem}[section]
\newtheorem{lemma}[theorem]{Lemma}
\newtheorem{proposition}[theorem]{Proposition}
\newtheorem{corollary}[theorem]{Corollary}
\newtheorem{definition*}{Definition}
\newtheorem*{problem*}{Problem}

\theoremstyle{definition}
\newtheorem{example}{Example}

\newtheorem{remark}{\textsc{Remark}}

\renewcommand\epsilon{\varepsilon}
\renewcommand\ss{\mathfrak{s}}
\renewcommand\tt{\mathfrak{t}}

\providecommand\llb{\llbracket}
\providecommand\rrb{\rrbracket}

\providecommand\gonc{Gon\v{c}arov{ }}

\providecommand{\dd}{\mathfrak{d}}

\providecommand{\LL}{\mathfrak{L}}
\providecommand{\EE}{\mathcal{E}}

\providecommand{\NN}{\mathbb{N}}

\providecommand{\bbx}{\mathbf{x}}

\newcommand{\fixed}[2][1]{%
  \begingroup
  \spaceskip=#1\fontdimen2\font minus \fontdimen4\font
  \xspaceskip=0pt\relax
  #2%
  \endgroup
}

\begin{document}
\begin{frontmatter}

\title{Multivariate Delta Gon\v{c}arov and Abel Polynomials}

\author{Rudolph Lorentz} 
\address{Department of Mathematics, Texas A\&M University at Qatar \\ PO Box 23874 Doha, Qatar}
\ead{rudolph.lorentz@qatar.tamu.edu}

\author{Salvatore Tringali\fnref{myfootnote}}
\address{Department of Mathematics, Texas A\&M University at Qatar \\ PO Box 23874 Doha, Qatar}
\fntext[myfootnote]{Current address: Institute for Mathematics and Scientific Computing, University of Graz,  Heinrichstr. 36, 8010 Graz, Austria.}
\ead{salvatore.tringali@uni-graz.at}

\author{Catherine H. Yan\corref{mycorrespondingauthor}}
\address{Department of Mathematics, Texas A\&M University \\ 77845-3368 College Station, TX.  USA}
\cortext[mycorrespondingauthor]{Corresponding author}
\ead{cyan@math.tamu.edu}

%

\begin{abstract}
Classical Gon\v{c}arov polynomials are polynomials which interpolate derivatives.
Delta Gon\v{c}arov polynomials are polynomials which interpolate
delta operators, e.g., forward and backward difference operators. We extend fundamental aspects of the theory of classical bivariate
Gon\v{c}arov polynomials and univariate delta Gon\v{c}arov polynomials
to the multivariate
setting using umbral calculus.
After introducing systems of delta operators,
we define multivariate delta Gon\v{c}arov polynomials,
show that the associated interpolation problem
is always solvable, and derive a generating function (an Appell
relation) for them. We show that systems of delta Gon\v{c}arov polynomials
on an interpolation grid $Z \subseteq \mathbb{R}^d$ are of binomial type if and only if  $Z = A\mathbb{N}^d$ for some $d\times d$ matrix $A$. This motivates
our definition of delta Abel polynomials to be exactly those delta Gon\v{c}arov
polynomials which are based on such a grid. Finally, compact formulas
for delta Abel polynomials
in all dimensions are given for separable systems of delta operators.
This recovers a former result for classical bivariate Abel polynomials and
extends previous partial results for classical trivariate Abel polynomials
to all dimensions.
\end{abstract}

\begin{keyword}
Abel and Gon\v{c}arov polynomials \sep Appell relations \sep
 delta operators \sep  interpolation \sep  umbral calculus.

 \MSC[2010] 41A05 \sep  05A40  \sep  33C45 \sep 41A10.
\end{keyword}


\end{frontmatter}



%
%
%
\section{Introduction}
\label{sec:intro}
The main purpose of this paper is to extend some fundamental aspects of the theory of Gon\v{c}arov
and Abel polynomials to higher dimensions and to replace
partial derivatives by systems of delta operators.

Such topics have an extensive history.
In 1881, Abel \cite{abel} introduced a sequence $g_0, g_1, g_2, \ldots$ of polynomials, now carrying his name, to
represent analytic functions. These polynomials are determined by the condition
that they interpolate the derivatives of \textit{any} given analytic function $f: \mathbb R \to \mathbb R$ at the nodes of an arithmetic progression
 through the formula
\begin{equation}
f(x) = \sum_{n=0}^\infty \frac{g_n(x)}{n!} f^{(n)}(nb),
\end{equation}
where $b \in \mathbb R$ is a fixed parameter.
In particular, it can be shown that $g_n(x) = x(x-nb)^{n-1}$ for every $n$, and $g_n$
satisfies the orthogonality condition
$$
g_n^{(k)}(kb) = n!\delta_{k,n}\qquad \text{for all }k.
$$

Abel polynomials count some basic combinatorial objects, for example,
 labeled trees and (generalized) parking functions \cite{[Yan14]}.
  The sequence $\{ g_n(x)\}_{n \geq 0}$  is  of binomial type,
  which leads to
a connection with umbral calculus (or finite operator calculus),
a branch of mathematics that studies analytic and algebraic-combinatorial properties of polynomials by a systematic use of  operator methods.

The (classical) umbral calculus of polynomials in one variable was put onto a firm theoretical basis by
Rota et al.~in a series of papers  \cite{Mullin-Rota, RoKaOd73, Rota}. It was extended
to multivariate polynomials and applied to combinatorial problems
in \cite{Garsia, Niederhausen, parrish, Reiner}. In \cite{Garsia-Joni, Joni78}
and \cite{Verde-Star}, higher dimensional umbral calculus is used to derive
various versions of the Lagrange inversion formula.

The operators considered in umbral calculus are delta
operators, a family of linear operators, acting on the algebra of univariate polynomials with coefficients in a field. Delta operators  share many properties in common with derivatives.    Each delta operator $\dd$  is uniquely associated with a sequence $\{p_n\}$ of polynomials of binomial type,
which interpolates the iterates of $\dd$, evaluated at $0$, as
$$
f(x)= \sum_{n=0}^\infty \frac{p_n(x)}{n!}  [\dd^n f(x)]_{x=0}.
$$

In \cite{Gonc30,[Gon]}, Gon\v{c}arov allowed the interpolation grid to be arbitrary,
obtaining that
\begin{equation}
f(x) = \sum_{n=0}^\infty g_n(x; a_0, a_1,\ldots , a_n) f^{(n)}(a_n),
\end{equation}
where $a_n\in \mathbb{R}$ and $g_n(x; a_0, a_1, \ldots, a_n) $
are the Gon\v{c}arov polynomials.
Such polynomials
have an extensive history in numerical analysis, even for interpolation
of derivatives \cite{Lo92}. Gon\v{c}arov polynomials
in two or more variables  have the unusual property that the interpolation
problem is solvable for any choice of the nodes of interpolation.
The uniqueness was  shown
in \cite{GloRlo} for the bivariate case and in \cite{JiaSharma} for
the multivariate case.

It is well known that certain values of (univariate) Gon\v{c}arov polynomials
are connected with order statistics (e.g., see \cite{[Kung]}).
 This connection  has been further developed \cite{[KY03]} into a complete correspondence between  Gon\v{c}arov polynomials and parking functions, a discrete structure lying at the heart of combinatorics.
In \cite{[KSY]},  difference Gon\v{c}arov polynomials were studied. Since difference operators are delta operators, this was another connection with umbral calculus.

Our basic goal here is to  extend the theory of classical
multivariate  Gon\v{c}arov polynomials by using umbral calculus and replacing partial derivatives with \textit{delta operators}.
This is a further development of our previous work \cite{[Khare-lorentz-yan],
 [Lorentz-yan]}  on  the analytic and combinatorial properties of  bivariate Gon\v{c}arov polynomials  and \cite{[Lorentz-tringali-yan]} on the
interpolation with general univariate delta operators.
In addition,
 the theory of generating functions, polynomial recursion
and approximation theory (the latter being Abel's original motivation
for studying his Abel polynomials), just to mention a few,
all play a role here.

The rest of the paper is organized as follows.
Section 2 contains  the  definition  and basic properties of a system of delta operators in $d$ variables.
In Section 3 we define the multivariate Gon\v{c}arov polynomials associated with a system of delta operators and an interpolation grid $Z$,
derive a  generating function (Appell relation), and characterize
the set of delta Abel polynomials, which are multivariate  Gon\v{c}arov polynomials of binomial type associated with delta operators. In the last two sections, we present  closed formulas for multivariate delta Abel polynomials in the special case when these are associated with a \textit{separable} system of delta operators. In particular, Section 4 deals with the
 bivariate case, and  Section 5 contains the general formulas in  an arbitrary dimension.

\section{Systems of  delta operators} 
%

A univariate delta operator $\dd$ is defined as a linear operator on the space of univariate polynomials in one variable $x$ that is
shift-invariant and for which $\dd(x)$ is a non-zero constant. In the multivariate case,
we need to consider \textit{systems of delta operators}.  To state the definition,
 we first give  some preliminary notations  using  multi-indices.

Let $d$ be a fixed integer $\ge 1$ (the space dimension).
For a vector $\mathbf{v} \in \mathbb{R}^d$,
we denote by $v_j$ the $j$-th component of $\mathbf{v}$. We write the vector of
space variables as $\mathbf{x} = (x_1, x_2, \ldots, x_d)$, and we do the same with the vectors of $\mathbb R^d$. In particular,
for $\mathbf{y}, \mathbf{z} \in \mathbb{R}^d$ we let $\mathbf{y} \cdot \mathbf{z} =\sum_{i=1}^d y_iz_i$, and we use $\mathbf{0}$ for the zero vector of $\mathbb R^d$.

Given $\mathbf{n} = (n_1, n_2, \ldots, n_d) \in \NN^d$, we set
$\mathbf{n}!= n_1 ! n_2 ! \cdots n_d !$, $| \mathbf{n}|=\sum_{i=1}^d n_i$, and
 ${\bbx}^{\mathbf{n}}=x_1^{n_1} x_2^{n_2} \cdots x_d^{n_d}$ (typically, the vectors of $\NN^d$ will serve as indices for the Gon\v{c}arov polynomials we will
work with).
 For $\mathbf{k}, \mathbf{n} \in \mathbb{N}^d$,
 $\mathbf{k} \leq \mathbf{n}$ means $k_i \leq n_i$ for all $i$, and
 $\binom{\mathbf{n}}{\mathbf{k}} = \binom{n_1}{k_1} \cdots \binom{n_d}{k_d}$.

Let $\mathbb{R}[\mathbf{x}]$ and $\mathbb{R}\llb \mathbf{x} \rrb$ be,
respectively, the algebra of polynomials (of finite degree) and the algebra of formal power
series in the variables $x_1, x_2, \ldots, x_d$ with coefficients in $\mathbb{R}$.
The (total) degree of a multivariate polynomial $p(\mathbf{x})$ is the maximum of $|\mathbf{n}|$ for all $\mathbf n \in \NN^d$ such that $\mathbf{x}^\mathbf{n}$ appears in $p(\mathbf{x})$ with a non-zero coefficient, while by the \textit{(relative) degree} of $p(\mathbf{x})$ \textit{with respect to $x_i$}
we mean the highest power $m$ for which $x_i^m$ is a factor of a monomial $\mathbf x^{\bf n}$
appearing in $p(\mathbf x)$ with a non-zero coefficient.


We denote by $\EE_{\mathbf{v}}$, for each $\mathbf{v} \in \mathbb{R}^d$, the shift operator $\mathbb{R}[\mathbf{x}]
\to \mathbb{R}[\mathbf{x}]: f(\mathbf{x}) \mapsto f(\mathbf{x}+\mathbf{v})$.
We say that a linear
operator $\mathfrak{L}: \mathbb{R}[\mathbf{x}] \to \mathbb{R}[\bbx]$  is
shift-invariant if $\mathfrak{L}\EE_{\mathbf{v}} =
\EE_{\mathbf{v}} \mathfrak{L}$ for every $\mathbf{v} \in \mathbb{R}^d$.
In particular, $\EE_\mathbf{v}$ is shift-invariant.

%
\begin{definition*} \label{def:system}
Let $\dd_1, \dd_2, \ldots, \dd_d$  be shift-invariant  linear operators on $\mathbb{R}[\bbx]$.
Then $(\dd_1, \dd_2, \ldots, \dd_d)$ is a \emph{system  of delta operators}  if for
any linear polynomial $p(\bbx) = a_1x_1+a_2 x_2 + \cdots +a_d x_d$, we have that
 $\dd_i p(x)=c_i$ for every $i = 1, \ldots, d$, where $c_i \in \mathbb R$ is a constant depending only on $i$, and at least one of  $c_1, c_2, \dots, c_d$ is not $0$.
\end{definition*}
It is shown by Parrish \cite[Theorem 4.1]{parrish} that every shift-invariant linear operator $\LL$ on
$\mathbb{R}[\mathbf{x}]$ is a formal power series in  $D_1, D_2, \dots, D_d$,  where
$D_i=D_{x_i}$ is the partial derivative with respect to the variable $x_i$.
To wit, we can write
\begin{equation}
\label{equ:Parrish_representation}
\LL =  \sum_{\mathbf{n} \in \mathbb{N}^d}  a_{\mathbf{n}} D_1^{n_1} D_2^{n_2} \cdots D_d^{n_d}.
\end{equation}
We refer to the formal power series
$$
 f(\mathbf{x})= \sum_{\mathbf{n} \in \mathbb{N}^d}  a_{\mathbf{n}} \mathbf{x}^\mathbf{n}
$$
as the indicator of $\LL$, and we write $\LL=f(D_1, D_2, \ldots, D_d)$. With this notation in hand, we have the following:
\begin{theorem}\label{admissible}
Let $\dd_1,\dd_2, \dots, \dd_d$ be shift-invariant
linear operators with indicators $f_1(\bbx)$, $f_2(\bbx), \dots$, $f_d(\bbx)$.
Assume
\begin{equation*}
 f_i(\bbx) = \sum_{\mathbf{n} \in \mathbb{N}^d}  a^{(i)}_{\mathbf{n}} \mathbf{x}^\mathbf{n}.
\end{equation*}
We form the $d \times d$ matrix  $J_f=(m_{i,j})$ by letting $m_{i,j}$ be the constant term of the formal power series $\partial f_i/\partial x_j$.
Then $(\dd_1, \dd_2, \ldots, \dd_d)$ is a system of delta operators if and only if
$\det J_f \neq 0$ and $a^{(i)}_{\mathbf{0}}= 0$ for all $i$.
\end{theorem}
\begin{proof}
Assume $(\dd_1,\dd_2,\ldots,\dd_d)$ is a system of delta operators, and
let $p(\mathbf{x}) = p_0 + \sum_{j=1}^d p_j x_j$, where $p_0, p_1, \ldots, p_d \in \mathbb R$.
The representation of delta operators provided by Eq.~\eqref{equ:Parrish_representation} yields that
\begin{equation} \label{dd-indicator}
\dd_i p(\mathbf{x}) = a^{(i)}_{\mathbf{0}}p(\mathbf{x}) + \sum_{j=1}^d m_{i,j} p_j \qquad (i = 1, \ldots, d).
\end{equation}
Also, $\dd_i f(\mathbf{x})$ being a constant for every linear polynomial $f(\mathbf{x})$ implies, for each $i$,
that $a^{(i)}_{\mathbf{0}}= 0$, and hence $\dd_i p(\mathbf{x})=\sum_{j=1}^d m_{i,j} p_j$.
But then, the condition that $\dd_i p(\mathbf{x}) \ne 0$ for some $i$ means that
there is no non-trivial solution to the linear system $ J_f \mathbf{v} =\mathbf{0}$. It follows that
$\det J_f \neq 0$.

Conversely, if $a_\mathbf{0}^{(i)} = 0$, then $\dd_i p(\mathbf{x})=\sum_{j=1}^d m_{i,j} p_j \in \mathbb{R}$ for all $i$.
And since $\det J_f \neq 0$, we have that  $J_f v \neq \mathbf{0}$ for any
non-zero vector  $\mathbf{v}= (p_1, p_2, \ldots, p_d) \in \mathbb{R}^d$.   Thus  any linear polynomial
of the form $p_0+\sum_{j=1}^d p_j x_j$
satisfies $\dd_i p(\mathbf x) \neq 0$ for some $i$.
\end{proof}	
\begin{definition*}
A set of polynomials $\{p_{\mathbf{n}}(\bbx)\}_{\mathbf{n} \in \mathbb{N}^d}$ is said to be a basic sequence of a system of delta operators $(\dd_1, \dd_2, \ldots, \dd_d)$ if:
\begin{enumerate}[label={\rm (\arabic{*})}]
\item the degree of $p_{\mathbf{n}}(\mathbf{x})$ is $|\mathbf{n}|$ for every $\mathbf n \in \mathbb N^d$;
\item
$p_{\mathbf{0}}(\mathbf{x}) = 1$ and
$p_{\mathbf{n}}(\mathbf{0}) = 0$ for all $\mathbf{n} \in \mathbb{N}^d$  with $|\mathbf{n}| \geq 1$;
\item
$\dd_i (p_{\mathbf{n}}) = n_i p_{\mathbf{n}-\mathbf{e}_i}$ for all
$\mathbf{n} \in \mathbb{N}^d$ and $1 \leq i \leq d$, where $\mathbf{e}_i$ is the $i$-th standard basis
vector $(0,\dots, 1, \dots, 0)$.
\end{enumerate}
\end{definition*}
From Parrish \cite{parrish},   each system  of delta operators
has a basic sequence (called normalized shift basis  in \cite{parrish}).
Moreover, a sequence of polynomials $\{ p_{\mathbf{n}}(\bbx) \}_{\mathbf{n} \in \mathbb{N}^d} $
is the basic sequence  of some system of delta operators if and only if it satisfies
the binomial identity
$$
p_{\mathbf{n}} (\mathbf{x} + \mathbf{y}) =\sum_{ \mathbf{k} \leq \mathbf{n}}
\binom{\mathbf{n}}{\mathbf{k}} p_{\mathbf{k}}(\mathbf{x}) p_{\mathbf{n}-\mathbf{k}} (\mathbf{y})
\qquad \text{ for all } \mathbf{n} \in \mathbb{N}^d.
$$
Notice that,
for each $m \in \mathbb{N}$, the set $\{ p_{\mathbf{n}}(\mathbf x): |\mathbf{n}| \leq m\}$ is a basis of the vector space of polynomials of $\mathbb R[\mathbf x]$ of degree at most $m$.

Univariate delta operators reduce the degree of a polynomial by one.
The next definition gives a generalization of  this property.
\begin{definition*} We say that the shift-invariant linear operators
$(\dd_1 , \dd_2, \dots, \dd_d)$ form  a \emph{strict system  of delta operators} if $\dd_i$ is degree reducing for each $i$, in the sense that for any polynomial  $p(\bbx)$ of degree $m$ with respect to $x_i$,
the degree of $\dd_i p $ with respect to $x_i$ is $m-1$ if $m \ge 1$, otherwise  $\dd_i(p) = 0$.
\end{definition*}
Strict systems of delta operators can be characterized as follows.

\begin{theorem}\label{strict-delta}
A system of shift-invariant delta operators $(\dd_1 , \dd_2, \dots, \dd_d)$ is a strict system if and only if
for each $i$,
the indicator  $f_i(\mathbf{x})$ of $\dd_i$ can be written as
$f_i(\bbx) = x_i g_i (\bbx)$ for some formal power series $g_i(\bbx)$ with $g_i(\mathbf{0})\neq 0$.
\end{theorem}
\begin{proof}
It is clear that if $f_i(\bbx) =  x_i g_i (\bbx)$ with $g_i(\mathbf{0})\neq 0$, then
the operator $\dd_i$ reduces the degree with respect to $x_i$ by 1.

Conversely, assume that  $(\dd_1, \dd_2, \dots, \dd_d)$   is a strict system of
 delta operators. We will show that $f_1(\bbx) =  x_1 g_1 (\bbx)$ with $g_1(\mathbf{0})\neq 0$.
 The formulas for other $i$  are similar.
Assume
$$
f_1=\sum_{\mathbf{n} \in \mathbb{N}^d} c_{\mathbf{n}} \bbx^{\mathbf{n}}.
$$
We need to show that $c_{\mathbf{n}}= 0$  whenever $n_1=0$,
 and $c_{(1,0,\dots, 0)} \neq 0$.

First, notice that $\dd_1(1)=0$, which implies $c_{\mathbf{0}}=0$.
Then, assume that $c_{\mathbf{n}}\neq  0$  for some $\mathbf{n}$ with  $n_1=0$.
Find in the set
$$
\{\mathbf{n} \in \mathbb N^d: n=(0, n_2, \dots, n_d) \text{ and  } c_{\mathbf{n}} \neq 0\}
$$
the minimal index $\mathbf{m}=(0, m_2, \dots, m_d)$ under the lexicographic order.
Apply $\dd_1 $ to $x_1 \cdot \bbx^{\mathbf{m}}$.  Note that
$D_1^{n_1} D_2^{n_2} \cdots D_d^{n_d} ( x_1 x_2^{m_2} \cdots x_d^{m_d})
$ is $0$ if $n_1 > 1$ or $n_i > m_i$ for any $2 \leq i \leq d$,   and is a polynomial with no
$x_1$ factor if $n_1=1$. With our choice of $\mathbf{m}$, we conclude that
$$
\dd_1(x_1 \cdot \bbx^{\mathbf{m}}) = c_{\mathbf{m}} m_2 !  \cdots m_d ! x_1
 + g(x_2, \dots, x_d),
 $$
  which is of degree 1 with respect to $x_1$, a contradiction.

 It remains to show that  $c_{(1,0,\dots, 0)} \neq 0$, which can be seen by applying
 $\dd_1$ to $x_1^2$.  Indeed, if $c_{(1,0,\dots, 0)} = 0$, then  $\dd_1(x_1^2)=2 c_{(2,0,\dots, 0)}$,
 so that $\dd_1$ has reduced the degree with respect to  $x_1$ by at least 2, which is a contradiction.
\end{proof}

\begin{remark}
A system of formal power series $f_1(\mathbf{x}), f_2(\mathbf{x}), \dots, f_d(\bbx)$ is called \emph{admissible}
if and only if (i) each $f_i(\bbx)$ has zero constant term and (ii)  $\det J_f \ne 0$.
(See e.g. \cite{Garsia-Joni, Joni78, Roman}.)
Hence, a system of delta operators in our definition is precisely
a system of shift-invariant linear operators whose indicators form an admissible system.
Such systems of operators have been called ``admissible differential operators'' by Joni \cite{Joni78}
and ``umbral operators'' by Garsia and Joni \cite{Garsia-Joni}.

It is known that for any admissible system $f_1(\mathbf{x}), f_2(\mathbf{x}), \dots, f_d(\bbx)$ there exists
a unique compositional inverse, i.e., an admissible system of formal power series
$g_1(\mathbf{x}), g_2(\mathbf{x}), \dots, g_d(\bbx)$
such that
$$
 f_i(g_1(\bbx), g_2(\mathbf{x}), \dots, g_d(\bbx) )=x_i\ \ \text{and}\ \
 g_i(f_1(\mathbf{x}), f_2(\mathbf{x}), \dots, f_d(\bbx))=x_i
$$
for all $i$. Consequently a system of delta operators has compositional inverse.
\end{remark}
\begin{remark}
Another  definition    of system of delta operators  is given
by Niederhausen \cite{Niederhausen}  for pairs of delta operators whose indicators are ``delta multi-series'', which is equivalent to a strict system of
delta operators defined here. Clearly, strict systems are special cases of systems of delta operators.
Conversely, not all admissible systems are strict. For example,  the pair
$(\dd_1, \dd_2)=(D_x+D_y,D_x-D_y)$ in the bivariate case  has admissible indicators, but does not form a  strict system.
\end{remark}
%
%

%
%
%
\section{Delta \texorpdfstring{Gon\v{c}arov}{Goncarov} polynomials}
In the present section we introduce delta Gon\v{c}arov polynomials starting from the bi\-or\-thog\-o\-nal\-i\-ty condition  in the Gon\v{c}arov interpolation problem. The univariate version was
discussed in  \cite{[KY03]} with the differentiation operator and
in \cite{[Lorentz-tringali-yan]} with an arbitrary delta operator. Here we extend the theory to many variables.

To this end, let $\Delta = (\dd_1, \dd_2, \dots, \dd_d)$ be a system of delta operators, which we assume as fixed for the rest of the section, and
let $\{\Phi_{\mathbf{n}}: \mathbf{n} \in \mathbb{N}^d \}$ be a set of shift-invariant operators of the form
\begin{equation}\label{Phi}
\Phi_{\mathbf{n}} = \LL_{\mathbf{n}} \dd_1^{n_1} \dd_2^{n_2} \cdots \dd_d ^{n_d},
\end{equation}
where $\LL_{\mathbf{n}}$ is a shift-invariant linear operator whose indicator  has a non-zero constant.

We say that a set of polynomials $\{q_{\mathbf{n}}(\bbx)\}_{\mathbf{n} \in \mathbb{N}^d}$ is a sequence of
polynomials if $\deg q_{\mathbf{n}} = |\mathbf n|$ for all $\mathbf{n}$. A sequence
$\{q_{\mathbf{n}}(\bbx) \}_{\mathbf{n} \in \mathbb{N}^d}$ of polynomials
is then called biorthogonal to the set
$\{\Phi_{\mathbf{n}}:  \mathbf{n} \in \mathbb{N}^d\} $ if
\begin{equation}
\epsilon(\mathbf{0}) \Phi_{\mathbf{n}}(q_{\mathbf{k}}(\bbx))  = \mathbf{n}! \delta_{\mathbf{n}, \mathbf{k}}.
\end{equation}

\begin{theorem} \label{biorthogonal}
Given a set $\{ \Phi_{\mathbf{n}}: \mathbf{n} \in \mathbb{N}^d \} $ of shift-invariant operators,
there exists  a unique sequence of polynomials that is biorthogonal to it.
\end{theorem}
To prove Theorem \ref{biorthogonal} we need some additional notation and a couple lemmas.
\begin{definition*}
 A subset $S$ of $\mathbb{N}^d$ is a lower set if for any $\mathbf{n} \in S$, it follows that
 $\mathbf{k} \in S$ for any $\mathbf{0} \leq \mathbf{k}\leq  \mathbf{n}$.
\end{definition*}
For any lower set $S$, we let
$$
\Pi^d_S(\Delta) = \left\{ q(\bbx):  q(\bbx) = \sum_{\mathbf{k} \in S}
 a_{\mathbf{k}} p_{\mathbf{k}}(\bbx), a_{\mathbf{k}} \in \mathbb{R} \right\},
$$
where $\{p_{\mathbf{n}}\}_{\mathbf{n} \in \mathbb{N}^d}$ is the basic sequence of $(\dd_1 , \dd_2, \dots, \dd_d)$.
In particular, we write
$$
\Pi^d_{\mathbf{n}}(\Delta) = \left\{ q(\bbx): q(\bbx) = \sum_{\mathbf{k} \leq \mathbf{n}}
 a_{\mathbf{k}} p_{\mathbf{k}}(\bbx), a_{\mathbf{k}} \in \mathbb{R} \right\}.
$$
\begin{lemma} \label{independence}
Let $S \subseteq \mathbb N^d$ a lower set. For any $\mathbf{n}  \in \mathbb{N}^d$, the set of functionals
\begin{equation}
\{  \Psi_{\mathbf{k} } = \epsilon(\mathbf 0) \Phi_{\mathbf{k}}:  \mathbf{k} \in \mathbb{N}^d,
\mathbf{k}  \leq \mathbf{n} \}
\end{equation}
is a linearly independent set over $\Pi^d_{\mathbf{n}}(\Delta)$.
\end{lemma}
\begin{proof}
Suppose there are real numbers $\lambda_{\mathbf{k}}$ such that
\begin{equation} \label{linear-1}
\Psi (p(\mathbf{x})) = \sum_{\mathbf{k} \leq \mathbf{n} }\lambda_{\mathbf{k} }\Psi_{\mathbf{k}} (p(\bbx)) =0
\end{equation}
for each $p(\bbx)  \in \Pi^d_{\mathbf{n}}(\Delta)$. We will show that the coefficients $\lambda_{\mathbf{k}}$ all vanish.

The basic sequence $\{ p_{\mathbf{n}}:  \mathbf{n} \in \mathbb{N}^d\}$
 has the property that
\begin{equation}
\dd_1^{k_1}  \dd_2^{k_2} \cdots \dd_d^{k_d}  p_{\mathbf{n}} = (n_1)_{k_1} (n_2)_{k_2} \cdots  (n_d)_{k_d} p_{\mathbf{n}-\mathbf{k}}
\end{equation}
for each $\mathbf{k} \leq \mathbf{n}$, and
\begin{equation} \label{vanish}
\dd_1^{k_1}  \dd_2^{k_2} \cdots \dd_d^{k_d}  p_{\mathbf{n}} = 0
\end{equation}
whenever $k_i > n_i$ for some $i$.

Substituting  $p_{\mathbf{0}}(\bbx) =1$ into Eq.~\eqref{linear-1}, we have
\begin{equation}
\Psi (p_{\mathbf{0}} (\mathbf{x})) =  \lambda_{\mathbf{0} }
\varepsilon(\mathbf{0}) \LL_{\mathbf{0}} (1)  =0.
\end{equation}
Since the indicator of $\mathfrak{L}_{\mathbf{0}}$ has non-zero constant
term, this yields $\LL_{\mathbf{0}} (1)  \neq 0$, and hence $\lambda_{\mathbf{0}}=0$.

Next we show that for any index vector $\mathbf{k} \in \mathbb{N}^d$ with $\mathbf{0} \leq \mathbf{k} \leq \mathbf{n}$,  if $\lambda_{\mathbf{s} } =0$ for all
$\mathbf{0} \leq \mathbf{s} \leq \mathbf{k}$ and $\mathbf{s} \neq \mathbf{k}$, then
$\lambda_{\mathbf{k}} =0$.
To see this, we substitute $p_{\mathbf{k}}(\bbx)$ into Eq.~\eqref{linear-1} and use the assumption
that $\lambda_{\mathbf{s}}=0$  for $\mathbf{s} < \mathbf{n}$
and Eq.~\eqref{vanish}, so as to find that
$$
\Psi(p_{\mathbf{k}}(\bbx)) = \lambda_{\mathbf{k}} \epsilon(\mathbf 0) \LL_{\mathbf{k}}( \mathbf{k}!).
$$
But the indicator of $\LL_{\mathbf{k}}$ has non-zero constant term, so
 $\LL_{\mathbf{k}}( \mathbf{k}!)$ is a non-zero constant. Evaluating at $\mathbf{0}$ does not change its value, and this implies $\lambda_{\mathbf{k}}=0$.

The lemma then follows from an induction on $|\mathbf{k}|$.
\end{proof}

\begin{lemma} \label{linear-2}
 Let $V$ be a  vector space of dimension $m$, and $\{ \tt_1, \tt_2, \dots, \tt_m\} $ an
 independent  set  of linear functionals on $V$. Then, for any
$b_1, \dots, b_m \in \mathbb R$ there is a unique vector $v \in V$ such that
 $ \tt_i(v) = b_i$ for all $i$.
 \end{lemma}
 \begin{proof}
 Let $\{e_1, e_2, \dots, e_m\}$ be a basis of $V$  and let $a_{i,j} = \tt_i (e_j)$.  Then the lemma just says that $Ax=b$ is uniquely solvable if and only if the matrix $A=(a_{i,j})$ has linearly independent rows.
  \end{proof}
Combining Lemmas \ref{independence} and \ref{linear-2}, we have the following:
\begin{corollary} \label{biorthogonal-cor}
Under the assumptions of Lemma \ref{independence}, given any multiset of real numbers
$\{b_{\mathbf{k}}:   \mathbf{k} \in \mathbb{N}^d, \ \mathbf{k} \leq \mathbf{n} \}$, there is a unique polynomial
 $q(\bbx) \in \Pi_{\mathbf{n}}^d(\Delta)$ such that, for all $ \mathbf{k} \leq \mathbf{n}$,
\begin{equation}
\varepsilon (\mathbf{0}) \Phi_{\mathbf{k}} (q(\bbx)) =
\varepsilon (\mathbf{0}) \LL_{\mathbf{k}}  \dd_1^{k_1} \dd_2^{k_2} \cdots \dd_d^{k_d}(q(\bbx)) =  b_{\mathbf{k}}.
\end{equation}
\end{corollary}
Theorem \ref{biorthogonal} follows from Corollary \ref{biorthogonal-cor} by  letting
 $q_{\mathbf{n}}(\bbx)$ be the unique polynomial in $\Pi_{\mathbf{n}}^d(\Delta)$
 that satisfies
\begin{eqnarray} \label{partial to all}
\varepsilon (\mathbf{0}) \Phi_{\mathbf{k}} (q_{\mathbf{n}} (\bbx)) = \mathbf{n} ! \delta_{\mathbf{n}, \mathbf{k}}
\end{eqnarray}
for all $\mathbf{k}\leq  \mathbf{n}$.
By the definition of $\Pi_{\mathbf{n}}^d(\Delta)$,  given any polynomial $r(\bbx) \in \Pi_{\mathbf{n}}^d(\Delta)$
and  any $\mathbf{k}'$ with
$k_i' > n_i $ for some $i$,   we have
$$
\dd_1^{k_1'} \dd_2^{k_2'} \cdots \dd_d^{k_d'} (r(\bbx)) =0.
$$
In particular,  $\epsilon(\mathbf{0}) \Phi_{\mathbf{k}'} (q_{\mathbf{n}}(\bbx)) = 0$, and this shows that
the biorthogonality condition \eqref{partial to all} holds for all $\mathbf{k} \in \mathbb{N}^d$.
On the other hand, it is clear that $q_\mathbf{n}(\bbx)$ is of degree $m=|\mathbf{n}|$, otherwise $q_\mathbf{n}(\bbx)$ would be a linear
combination of $\{p_{\mathbf{k}}: |\mathbf{k}| < m\}$, and hence we would have $\epsilon (\mathbf{0})\Phi_{\mathbf{n}}(q_\mathbf{n}(\bbx))=0$,
which is a contradiction.

Putting it all together, we see from the above that the sequence of polynomials $\{ q_{\mathbf{n}}(\bbx)\}_{\mathbf{n} \in \mathbb{N}^d}$
is biorthogonal to the set of operators $\{ \Phi_{\mathbf{n}}: \mathbf{n} \in \mathbb{N}^d\}$.

In particular, if  each $\LL_{\mathbf{n}}$ in Eq.~\eqref{Phi} is a shift operator, then the set of polynomials biorthogonal to $\{ \Phi_{\mathbf{n}}: \mathbf{n} \in \mathbb{N}^d\}$ is just the set of the delta Gon\v{c}arov polynomials in which we are interested.
More precisely, using that $\epsilon(\mathbf{0})\EE_{\mathbf{v}} = \epsilon(\mathbf{v})$, we have the following:
\begin{definition*}
Let $\Delta = (\dd_1, \dd_2, \dots, \dd_d)$ be a system of delta operators, and
$Z=\{z_{\mathbf{k}}: \mathbf{k} \in S\} \subseteq \mathbb{R}^d$ a multiset of nodes,
which we refer to as an interpolation grid.
The delta Gon\v{c}arov polynomial  $t_{\mathbf{n}}(\bbx; Z)$
is the unique polynomial in $\Pi_{\mathbf{n}}^d(\Delta)$ satisfying
\begin{equation}\label{interpolation-condition}
\varepsilon (z_{\mathbf{k}})
\dd_1^{k_1}  \dd_2^{k_2} \cdots \dd_d^{k_d} (t_{\mathbf{n}}(\bbx;Z)) = \mathbf{n}!  \delta_{\mathbf{k}, \mathbf{n}}
\end{equation}
for all $\mathbf{k}  \leq \mathbf{n}$.
\end{definition*}
From the preceding
discussion, we see that $t_{\mathbf{n}}(\bbx; Z)$ only depends on the set of nodes
in $Z_{\mathbf{n}}= \{z_{\mathbf{k}}: \mathbf{k} \leq \mathbf{n}\}$.

The delta Gon\v{c}arov polynomials form a basis for the solutions of the following  interpolation
problem with a system of delta operators $\Delta = (\dd_1, \dd_2, \dots, \dd_d)$,
which we call the \textit{delta Gon\v{c}arov interpolation problem in dimension $d$}:
\begin{problem*}
Let $S$ be a lower set.
Given a multiset of nodes $Z=\{z_{\mathbf{k}}: \mathbf{k} \in S\} \subseteq \mathbb{R}^d$
and a multiset of real numbers $\{ b_{\mathbf{k}}: \mathbf{k} \in S\}$,
find a polynomial $P(\bbx)$
in $\Pi^d_S(\Delta)$ such that, for all $\mathbf{k} \in S$,
$$
\epsilon (z_{\mathbf{k}}) \dd_1^{k_1}  \dd_2^{k_2} \cdots \dd_d^{k_d}
(P(\bbx))  = b_{\mathbf{k}}.
$$
\end{problem*}
By Theorem \ref{biorthogonal}, the above problem has the unique solution
\begin{equation*}
P(\bbx) = \sum_{\mathbf{k} \in S}
\frac{b_{\mathbf{k}}}{\mathbf{k}!}t_{\mathbf{k}}(\bbx; Z).
\end{equation*}
Conversely, any polynomial in $\Pi^d_S(\Delta)$ can be expanded as a linear combination of delta Gon\v{c}arov polynomials.
This is the content of the next proposition, which follows from the orthogonality relations.
\begin{proposition} \label{expansion}
For any $P \in \Pi^d_S(\Delta)$,
\begin{equation*}
P(\bbx) = \sum_{\mathbf{k} \in S }
 \frac{1}{\mathbf{k}!}\left[ \varepsilon(z_{\mathbf{k}})\dd_1^{k_1} \dd_2^{k_2}
 \cdots \dd_d^{k_d} (P(\bbx)) \right]t_{\mathbf{k}}(\bbx; Z).
\end{equation*}
\end{proposition}
In particular, substituting $P(\bbx)$ in Proposition \ref{expansion} with the basic polynomial $p_{\mathbf{n}}(\bbx)$, and noting that $\binom{\mathbf{n}}{\mathbf{k}} =0$ unless
$\mathbf{k} \leq \mathbf{n}$,
we obtain a binomial relation
between the basic sequence and the delta Gon\v{c}arov polynomials,  which
generalizes the linear recurrence for univariate Gon\v{c}arov polynomials
\cite[Eq.~(2.5)]{[KY03]}.

\begin{proposition} 
Let $\{p_{\mathbf{n}}(\bbx)\}_{\mathbf{n}\in \mathbb{N}^d}$  be the basic sequence of the system of delta operators $(\dd_1, \dd_2,\ldots, \dd_d)$.
Then
\begin{equation} \label{linear-recursion}
p_{\mathbf{n}}(\bbx) = \sum_{ \mathbf{k} \leq \mathbf{n} }  \binom{\mathbf{n}}{\mathbf{k}}
p_{\mathbf{n}-\mathbf{k}} (z_{\mathbf{k}})t_{\mathbf{k}}(\bbx; Z).
\end{equation}
\end{proposition}
%
%
%
%
\section{Sequences of binomial type}
The basic sequence of a system of delta operators is of binomial type.
In general, this is not true for multivariate delta Gon\v{c}arov polynomials.
In the present section we give a necessary and sufficient condition under which
a sequence of  delta Gon\v{c}arov polynomials  is of binomial type.

The following definition is adapted from Parrish \cite{parrish}.
\begin{definition*}
A polynomial sequence $\{q_{\mathbf{n}}(\bbx)\}_{\mathbf{n} \in \mathbb{N}^d}$ is
 said to be of binomial type if $q_{\mathbf{0}},  q_{\mathbf{e}_1}, \ldots, q_{\mathbf{e}_d}$
 span the vector space of linear polynomials in $\mathbb{R}[\bbx]$ and, for
 any $\mathbf{n} \in \mathbb{N}^d$, $q_{\mathbf{n}}$ satisfies the identity
\begin{equation*}
q_{\mathbf{n}}(\bbx + \mathbf{y}) = \sum_{ \mathbf{k} \leq \mathbf{n}}
\binom{\mathbf{n}}{\mathbf{k}} q_{\mathbf{k}}(\bbx) q_{\mathbf{n}-\mathbf{k} }(\mathbf{y}).
\end{equation*}
\end{definition*}
In particular, Parrish showed in \cite{parrish} that:
\begin{enumerate}[label={(\arabic{*})}]
\item
A sequence of polynomials is  of binomial type if and only if
it is the basic sequence of a system of delta operators;
\item\label{item:Parrish_(2)}
A sequence of polynomials $\{ p_{\mathbf{n}}(\bbx)\}_{\mathbf{n} \in \mathbb{N}^d}  $ is a sequence of binomial type if and only if
\begin{eqnarray} \label{basic-GF}
\sum_{\mathbf{n} \in \mathbb{N}^d}   p_{\mathbf{n}}(\bbx)
\frac{\mathbf{y}^{\mathbf{n}}}{\mathbf{n}!}
= e^{\mathbf{x} \cdot \mathbf{g}(\mathbf{y})}
\end{eqnarray}
for an admissible system of formal  power series  $\mathbf{g}=(g_1, g_2, \dots, g_d)$.
\item  If $\{ p_{\mathbf{n}}(\bbx)\}_{\mathbf{n} \in \mathbb{N}^d}  $
 is the basic sequence of a system
$(\dd_1, \dd_2, \dots, \dd_d)$ of delta operators,
then the formal power series $g_1, g_2, \dots, g_d$ in point \ref{item:Parrish_(2)} above are, respectively, the compositional
inverses of $f_1, f_2, \dots, f_d$, where $f_i$ is the indicator of $\dd_i$ for each $i$. Consequently, Eq.~\eqref{basic-GF} is equivalent to
the following \textit{Appell relation} for the sequence $\{p_{\mathbf{n}}(\bbx)\}_{\mathbf{n} \in \mathbb{N}^d}$:
\begin{equation} \label{Appell-basic}
\sum_{\mathbf{n} \in \mathbb{N}^d} \frac{ p_{\mathbf{n}}(\bbx)}{\mathbf{n}!}
\mathbf{f}(\mathbf{y}) ^{\mathbf{n}}= e^{\mathbf{x} \cdot \mathbf{y}},
\end{equation}
where $\mathbf{f}(\mathbf{y}) =( f_1(\mathbf{y}), f_2(\mathbf{y}), \dots, f_d(\mathbf{y}))$.
\end{enumerate}
%
%
%

%
Now, fix a system of delta operators, and let their indicators and the corresponding compositional inverses be as in the above.
Then, we can derive an Appell relation for multivariate delta Gon\v{c}arov polynomials as follows:
\begin{theorem}[\textbf{Appell relation}] \label{delta-Appell}
Let $\{t_{\mathbf{n}}(\bbx;Z)\}_{\mathbf{n} \in \mathbb{N}^d} $ be the sequence of delta
Gon\v{c}arov polynomials associated to a system $\Delta = (\dd_1, \dd_2, \ldots, \dd_n)$ of delta operators and a grid $Z=\{ z_{\mathbf{n}}: \mathbf{n} \in \mathbb{N}^d \} \subseteq \mathbb{R}^d$. Then
\begin{equation}
\label{equ:Appell_relation_for_delta_Goncarov_polys}
e^{\bbx \cdot \mathbf{g}(\mathbf{y})}
=\sum_{\mathbf{k} \in \mathbb{N}^d}\frac{ t_{\mathbf{k}}(\bbx;Z)}{\mathbf{k}!} \mathbf{y}^{\mathbf{k}}
e^{z_{\mathbf{k}} \cdot \mathbf{g}(\mathbf{y})}.
\end{equation}
\end{theorem}
\begin{proof}
Combining Eqs.~\eqref{linear-recursion} and \eqref{basic-GF}, we have
\begin{equation*}
\begin{split}
e^{\bbx \cdot \mathbf{g}(\mathbf{y})}
= \sum_{\mathbf{n} \in \mathbb{N}^d}   p_{\mathbf{n}}(\bbx)
\frac{\mathbf{y}^{\mathbf{n}}}{\mathbf{n}!}
    & = \sum_{\mathbf{n} \in \mathbb{N}^d} \frac{1}{\mathbf{n}!} \left(
\sum_{ \mathbf{k} \leq \mathbf{n} }  \binom{\mathbf{n}}{\mathbf{k}}
p_{\mathbf{n}-\mathbf{k}} (z_{\mathbf{k}})t_{\mathbf{k}}(\bbx; Z) \right)
\mathbf{y}^{\mathbf{n}} \\
& = \sum_{\mathbf{k} \in \mathbb{N}^d } \left(
 \frac{1}{\mathbf{k}!}   t_{\mathbf{k}}(\bbx; Z)  \mathbf{y}^{\mathbf{k}}
 \cdot \sum_{ \mathbf{m}=\mathbf{n}-\mathbf{k}}
 \frac{p_{\mathbf{m}}(z_{\mathbf{k}}) }{\mathbf{m}!} \mathbf{y}^{\mathbf{m}} \right)  \\
 & = \sum_{\mathbf{k} \in \mathbb{N}^d }
 \frac{ t_{\mathbf{k}}(\bbx; Z)    }{\mathbf{k}!}    \mathbf{y}^{\mathbf{k}}
  e^{z_{\mathbf{k}} \cdot \mathbf{g}(\mathbf{y})},
\end{split}
\end{equation*}
which concludes the proof.
\end{proof}
Note that an equivalent version of the Appell relation \eqref{equ:Appell_relation_for_delta_Goncarov_polys} is
\begin{equation} \label{alternate}
e^{\mathbf{x} \cdot \mathbf{y}} =
 \sum_{\mathbf{k} \in \mathbb{N}^d } \left(
 \frac{ t_{\mathbf{k}}(\bbx; Z)    }{\mathbf{k}!}   \mathbf{f}(\mathbf{y})^{\mathbf{n}}
 \cdot e^{z_{\mathbf{k} } \cdot \mathbf{y}} \right).
 \end{equation}
This will be used in the proof of the following theorem, which is the main result of the present section and
characterizes the delta Gon\v{c}arov polynomials that are
of binomial type in terms of the ``geometry'' of the underlying interpolation grid, by showing that the latter is actually
a linear transformation of the set of lattice points  $\mathbb{N}^d$. We will refer to such \gonc polynomials as \textit{delta Abel polynomials}, as in the univariate case they are exactly the Abel polynomials
$x(x-nb)^{n-1}$, for which the interpolation grid is an arithmetic progression of the form $0, b, 2b, \ldots$ for some $b \in \mathbb R$.
\begin{theorem} \label{Binomial type}
Given a system of delta operators $(\dd_1, \dd_2, \dots, \dd_d)$ and a grid
$Z=\{ z_{\mathbf{k}}: \mathbf{k} \in \mathbb{N}^d \} \subseteq \mathbb{R}^d$.
Then the associated set of delta Gon\v{c}arov polynomials
  $\{t_{\mathbf{n}}(\bbx;Z)\}_{\mathbf{n} \in \mathbb{N}^d} $ is of binomial type if and only if
there is a $d \times d$ matrix $A$ for which
$z_{\mathbf{k}} = A \mathbf{k}$ for all $\mathbf{k}$. (Here we view
$\mathbf{k} \in \mathbb{N}^d$ as a column vector). In other words, the set of nodes is the image of
the set of lattice  points $\mathbb{N}^d$ under the matrix $A$.
\end{theorem}
\begin{proof}
We split the proof into two parts.
\vskip 0.1cm
\noindent {\bf Necessity.} Assume the sequence of delta Gon\v{c}arov  polynomials
$\{ t_{\mathbf{k}}(\bbx; Z)\}$ is of binomial type.
This is equivalent to the existence of
an admissible system of formal power series $\mathbf{g}=(g_1, g_2, \dots, g_d)$ such that
\begin{eqnarray} \label{binomial-1}
\sum_{\mathbf{k} \in \mathbb{N}^d}   t_{\mathbf{k}}(\bbx;Z)
\frac{\mathbf{y}^{\mathbf{k}}}{\mathbf{k}!}
= e^{\bbx \cdot \mathbf{g}(\mathbf{y})}.
\end{eqnarray}

Let $\mathbf{h}=(h_1,h_2, \dots, h_d) \in \mathbb{R} \llb x \rrb $  be the compositional inverse of
$\mathbf{g}$.  Substituting $\mathbf{y}$ with $\mathbf{h}(\mathbf{y})= (h_1(\mathbf{y}), h_2(\mathbf{y}), \dots,
h_d(\mathbf{y}))$, we obtain
\begin{eqnarray} \label{binomial-2}
\sum_{\mathbf{k} \in \mathbb{N}^d}  \frac{  t_{\mathbf{k}}(\bbx;Z)}{\mathbf{k}!}
 \mathbf{h}( \mathbf{y})^{\mathbf{k}}
= e^{\bbx \cdot \mathbf{y}}.
\end{eqnarray}
Comparing this with the alternate Appell relation  \eqref{alternate}
and using the uniqueness of the coefficients,
we have
\begin{equation} \label{h-and-k}
\mathbf{h}(\mathbf{y})^{\mathbf{k}} = \mathbf{f}(\mathbf{y})^{\mathbf{k}} \cdot e^{z_{\mathbf{k}} \cdot \mathbf{y}}
\qquad \text{ for all } \mathbf{k} \in \mathbb{N}^d .
\end{equation}
Taking the same equation with index $\mathbf{k} + \mathbf{e}_j$
 and dividing by the first we get
$$
h_j(\mathbf{y}) = f_j(\mathbf{y}) \cdot e^{(z_{\mathbf{k}+\mathbf{e}_j} - z_{\mathbf{k}}) \cdot \mathbf{y}},
$$
which  holds for all $\mathbf{k} \in \mathbb{N}^d$ and $1\leq j\leq d$.
 It follows that $z_{\mathbf{k}+\mathbf{e}_j} - z_{\mathbf{k}}  $ is
a constant vector independent of $\mathbf{k}$. Denote this constant vector by $\mathbf{t}_j$.
It follows that
$$
z_{\mathbf{k}}= z_{\mathbf{0}} + k_1 \mathbf{t}_1 + k_2 \mathbf{t}_2+ \cdots + k_d \mathbf{t}_d.
$$
So, taking $\mathbf{k}=\mathbf{0}$ in Eq.~\eqref{h-and-k}, we get
$e^{z_{\mathbf{0}} \cdot \mathbf{y}  }=1$,
which implies $z_{\mathbf{0}}=\mathbf{0}$. Hence $z_{\mathbf{k}}$ is of the form $A \mathbf{k}$,
where $A$ is the matrix whose columns are $\mathbf{t}_1, \mathbf{t}_2, \ldots, \mathbf{t}_d$.
\vskip 0.1cm
\noindent {\bf Sufficiency.}
Conversely, assume  $Z$ is a linear transformation of  $\mathbb{N}^d$ by a $d \times d$ matrix $A$.  Then
$$
z_{\mathbf{k}}=  k_1 \mathbf{t}_1 + k_2 \mathbf{t}_2+ \cdots + k_d \mathbf{t}_d,
$$
where $\mathbf{t}_1, \mathbf{t}_2, \ldots, \mathbf{t}_d$ are the column vectors of
 the matrix $A$.
The Appell relation for $t_{\mathbf{k}}(\bbx)$ then  becomes
\begin{equation} \label{Appell-2}
\begin{split}
e^{\mathbf{x} \cdot \mathbf{y}}
    = & \sum_{\mathbf{k} \in \mathbb{N}^d } \left(
 \frac{ t_{\mathbf{k}}(\bbx; Z)    }{\mathbf{k}!}   \mathbf{f}(\mathbf{y})^{\mathbf{k}}
 \cdot e^{ ( k_1 \mathbf{t}_1 + k_2 \mathbf{t}_2+ \fixed[0.2]{ \text{ }}\cdots\fixed[0.2]{ \text{ }} + k_d \mathbf{t}_d)  \cdot \mathbf{y}} ) \right)  \nonumber  \\
    =&\sum_{\mathbf{k} \in \mathbb{N}^d } \left(
  \frac{ t_{\mathbf{k}}(\bbx; Z)    }{\mathbf{k}!}   \mathbf{h}(\mathbf{y})^{\mathbf{k}} \right),
\end{split}
\end{equation}
where we set $\mathbf{h}(\mathbf{y}) =(h_1(\mathbf{y}), h_2(\mathbf{y}), \dots, h_d(\mathbf{y}))$,
 with
 $$
 h_i(\mathbf{y}) = f_i (\mathbf{y}) e^{\mathbf{t}_i \cdot \mathbf{y}}.
 $$
Since Eq.~\eqref{Appell-2} agrees with Eq.~\eqref{Appell-basic}, it only remains to
show that the system  $\mathbf{h}(\mathbf{y})$  is admissible.
For this, note that
$$
\frac{\partial h_i}{\partial y_j} = \frac{ \partial f_i}{\partial y_j} \cdot e^{\mathbf{t}_i \cdot \mathbf{y}} + f_i(\mathbf{y}) t_{i,j} e^{\mathbf{t}_i \cdot \mathbf{y}}.
$$
But $f_i(\mathbf{0})=0$, hence $\frac{\partial h_i}{\partial y_j}(\mathbf{0})=
\frac{\partial f_i}{\partial y_j} (\mathbf{0})$, and then
$J_h =J_f$  is non-singular.
Thus the sequence $\{t_{\mathbf{n}}(\bbx; Z)\}_{\mathbf{n} \in \mathbb{N}^d}$ is of binomial type.
\end{proof}
Given a system of delta operators, Theorem \ref{Binomial type} says that  if the defining grid  is a linear transformation of $\mathbb{N}^d$,
then  the associated sequence of delta Gon\v{c}arov polynomials
is of binomial type. Therefore, these polynomials are the basic sequence of some system of delta operators.
More explicitly, we have:
\begin{proposition} \label{shift-prop}
Let $\Delta = (\dd_1, \dd_2,\dots, \dd_d)$ be a system of delta operators. Let $Z$ be a linear transformation of $\mathbb{N}^d$ by a
$d \times d$ matrix $A$, and $\mathbf{t}_i$ the $i$-th column vector of $A$.
Denote by $\{ t_{\mathbf{n}}(\bbx;Z) \}_{\mathbf{n} \in \mathbb{N}^d}$ the sequence of delta Gon\v{c}arov
polynomials associated with $\Delta$ and $Z$.
Then $\{ t_{\mathbf{n}}(\bbx;Z)\}_{\mathbf{n} \in \mathbb{N}^d} $ is the basic sequence of the system of delta operators
$(\ss_1, \ss_2, \dots, \ss_d)$, with $\ss_i =\dd_i \EE_{\mathbf{t}_i}$ for all $i$.
\end{proposition}
This can be seen by writing down the orthogonality
conditions for the system $(\ss_1, \ss_2, \dots, \ss_d)$, or from the  proof of Theorem \ref{Binomial type}.
\section{Closed formulas for bivariate delta Abel polynomials}
In the rest of the paper  we will show how to derive compact formulas for  a family of
multivariate delta Abel polynomials using the properties  of delta operators.

The systems of delta operators considered here are
\textit{separable}, in the sense that
$\dd_i = D_i \LL_i$ for each $i$, where $\LL_i$ is an invertible  shift-invariant operator acting only on $x_i$, i.e.,
the indicator of $\LL_i$ is a formal power series in the variable $x_i$ only.
The classical  case where $\dd_i=D_i$ is a special example in this family.

We shall give a closed formula for the multivariate delta Abel polynomials, which are
delta Gon\v{c}arov polynomials  whose defining  grid  is a linear transformation
of  $\mathbb{N}^d$ by a $d\times d$ matrix.
From Theorem \ref{Binomial type}, these are exactly the delta Gon\v{c}arov polynomials
that  are of binomial type.

For simplicity and clarity, in this section we describe the approach and the results in
the bivariate case, where the variables are $x$ and $y$.
The general case is dealt with in a similar manner in the next section.

In \cite{[Lorentz-yan]}, it was conjectured and then verified that, for a $2\times 2$ matrix
\begin{equation}\label{linear grid matrix}
A =
\left(
\begin{array}{rr}
a_{11} & a_{12}\\
a_{21} & a_{22}
\end{array}
\right)\!,
\end{equation}
the classical bivariate Abel polynomials
on the grid
\begin{equation}\label{linear grid}
Z =\{(x_{i,j}, y_{i,j})\}_{(i,j) \in \mathbb N^2},\ \ \text{where }(x_{i,j}, y_{i,j}) = A (i, j)^t \ \ \text{for all }i, j \in \mathbb N,
\end{equation}
are given by
\begin{equation} \label{classical}
\left[ (x-x_{0,n})(y-y_{m,0}) - x_{0,n}y_{m,0}\right]
(x-x_{m,n})^{m-1}(y-y_{m,n})^{n-1}.
\end{equation}
%
We will derive this formula using the techniques of delta operators. For that we will need
the concept of bivariate Pincherle derivatives and their properties, see \cite{Mullin-Rota, RoKaOd73}.
These operators were also called ``Lie derivatives'' in \cite{Garsia-Joni, Joni78}.

Let $\LL$ be a shift-invariant operator on $\mathbb{R}[x,y]$.
Then the Pincherle derivatives of $\LL$ relative $x$ and $y$ are, respectively, the linear operators defined by taking, for every $p(x,y) \in \mathbb R[x,y]$,
\begin{equation*}
\LL_x'(p(x,y)) = \LL(x p(x,y)) -x\LL(p(x,y))
\end{equation*}
and
\begin{equation*}
\LL_y'(p(x,y)) = \LL(y p(x,y)) -y\LL(p(x,y)).
\end{equation*}
If $\LL$ is shift-invariant, then so are $\LL_x'$ and $\LL_y'$.
In addition, the indicator function of $\LL_x$ (resp. $\LL_y$) is $ \partial f(x,y)/\partial x$
(resp. $\partial f(x,y)/\partial y$), where $f(x,y)$ is the indicator of $\LL$.
Consequently, we have:
\begin{enumerate}
\item
$(D_x)_x' = \mathfrak{I}$, the identity operator,  and  $(D_x)_y' = 0$;
\item
$(\LL \mathfrak{S})_x' = (\LL)_x'\mathfrak{S} + \LL(\mathfrak{S})_x'$ for any shift-invariant operator $\mathfrak{S}$;
\item
$(\LL D_x)_x' = \LL + (\LL)_x' D_x$;
\item
$(\EE_{(a,b)})_x' = a\EE_{(a,b)}$, $(\EE_{(a,b)})_y' = b\EE_{(a,b)}$.
\end{enumerate}
%
Two formulas from \cite{RoKaOd73} we will be using are based on the fact
that any univariate delta operator can be written as $\dd = D\LL$,
where $\LL$ is an invertible shift-invariant operator. Let $\dd = D\LL$ be a delta
operator with basic sequence $\{p_n\}_{n=0}^\infty$. Then
\begin{equation}\label{pincherle equation}
p_n(x) = \dd' \LL^{-n-1}(x^n)
\end{equation}
and
\begin{equation}\label{power equation}
p_n(x) = x\LL^{-n}(x^{n-1}).
\end{equation}
The crucial step enabling the use of delta operators is Proposition \ref{shift-prop}, which
allows us to solve a Gon\v{c}arov interpolation problem on the type of linear  grid used here
by  determining a basic sequence.
\begin{theorem}\label{bivariate Abel}
Let $(\dd_x, \dd_y)$ with $\dd_x = D_x\LL_x$, $\dd_y =  D_y\LL_y$
be a separable pair of delta operators, and let $Z$ be a linear transformation of $\mathbb{N}^2$ by a $2 \times 2$ matrix $A$ as in
Eqs.~\eqref{linear grid matrix} and \eqref{linear grid}. Denote by $(t_{m,n}((x,y); Z))_{(m,n) \in \mathbb N^2}$ the sequence of Abel polynomials associated with the pair $(\dd_x, \dd_y)$
and the grid $Z$, and fix $m, n \in \mathbb N$. The following holds:
\begin{enumerate}[label={\rm(\roman{*})}]
\item\label{item:theorem_explicit_formulas(i)} Let $p_m$ and $q_n$ be the (univariate) basic sequences of $\dd_x$ and $\dd_y$ respectively. Then $t_{m,n}((x,y); Z)$ is given by
\begin{equation*}\label{nodes}
\left|
\begin{array}{cc}
\! x - x_{0,n} & y_{m,0}\!\\
\!x_{0,n} & y - y_{m,0}\!
\end{array}
\right|
\frac{p_m(x-x_{m,n})}{x-x_{m,n}}\cdot\frac{q_n(y-y_{m,n})}{y-y_{m,n}},
\end{equation*}
where $x_{i,j}=a_{1,1}i+a_{1,2} j $ and $y_{i,j}=a_{2,1} i + a_{2,2} j$ for all $0\leq i \leq m$ and $0 \leq j \leq n$.
\item\label{item:theorem_explicit_formulas(ii)} Let  $\ss_1 = D_x\LL_x\EE_{(a_{1,1}, a_{2,1})}$ and $\ss_2 = D_y\LL_y\EE_{(a_{1,2},a_{2,2})}$. Then
\begin{equation}\label{ss_1}
t_{m,n}((x,y); Z) =
J(\ss_1, \ss_2) \fixed[0.3]{ \text{ }}
\EE_{a_{1,1}, a_{2,1}}^{-m-1} \EE_{a_{1,2},a_{2,2}}^{-n-1}
\LL_x^{-m-1} \LL_y^{-n-1}(x^m y^n),
\end{equation}
where
\begin{equation}
\label{equ:jacobian_with_sij}
J(\ss_1, \ss_2) =
\left|
\begin{array}{cc}
\! (\ss_1)_x' & (\ss_1)_y' \!\\
\! (\ss_2)_x' & (\ss_2)_y' \!
\end{array}
\right|.
\end{equation}
\item\label{item:theorem_explicit_formulas(iii)} $t_{m,n}((x,y); Z) = T_{m,n}(x^m y^n)$, where $T_{m,n}$ is the operator
\begin{equation}
\label{equ:operator_Tmn}
\left|
\begin{array}{cc}
\! (\dd_x)_x' +a_{1,1}\dd_x  &  a_{2,1}\dd_x \!\\
\! a_{1,2}\dd_y  &  (\dd_y)_y' +a_{2,2}\dd_y \!
\end{array}
\right|
\EE_{(a_{1,1}, a_{2,1})}^{-m} \EE_{(a_{1,2},a_{2,2})}^{-n}
\LL_x^{-m-1} \LL_y^{-n-1}.
\end{equation}
\end{enumerate}
\end{theorem}
\begin{proof}
The proof consists of first showing that \ref{item:theorem_explicit_formulas(ii)} and \ref{item:theorem_explicit_formulas(iii)} are equivalent.
Then we prove that \ref{item:theorem_explicit_formulas(iii)} yields the basic sequence of the pair
$(\ss_1, \ss_2)$, which, by Proposition \ref{shift-prop},  is the same as
the bivariate Abel polynomial  associated with the pair $(\dd_x,\dd_y)$ and the grid $Z$.
Lastly, we use \eqref{pincherle equation} and \eqref{power equation} to derive \ref{nodes}.
\vskip 0.1cm
\noindent \textbf{Step 1.} \ref{item:theorem_explicit_formulas(ii)} and \ref{item:theorem_explicit_formulas(iii)}  are equivalent. We calculate
\begin{equation*}
\begin{split}
(\ss_1)_x' & = \EE_{(a_{1,1},a_{2,1})} [(\dd_x)_x' + a_{1,1}\dd_x]; \\
(\ss_1)_y' & = a_{2,1} \EE_{(a_{1,1}, a_{2,1})}  \dd_x; \\
(\ss_2)_x' & = a_{1,2} \EE_{(a_{1,2}, a_{2,2})}  \dd_y; \\
(\ss_2)_y' & = \EE_{(a_{1,2},a_{2,2})}[(\dd_y)_y' + a_{2,2}\dd_y].
\end{split}
\end{equation*}
Substituting these values into \eqref{ss_1} and \eqref{equ:jacobian_with_sij},
we obtain \ref{item:theorem_explicit_formulas(iii)}.
\vskip 0.1cm
\noindent \textbf{Step 2. \ref{item:theorem_explicit_formulas(iii)}} yields the basic sequence of the pair $(\ss_1, \ss_2)$.
Let, for $m, n \ge 0$,
$$
s_{m,n}(x,y) = T_{m,n}(x^m y^n).
$$
By induction, it is enough to show that $s_{m,n}$  satisfies the following:
\begin{enumerate}[label={\rm(\alph{*})}]
\item\label{item:properties_of_smn(a)} $s_{0,0}(x,y) = 1$;
\item\label{item:properties_of_smn(b)} $s_{m,n}(0,0) = 0$ if $m\geq 1$ or $n\geq 1$;
\item\label{item:properties_of_smn(c)} $\ss_1(s_{m,n}) = ms_{m-1,n}$, $\ss_2(s_{m,n}) = ns_{m,n-1}$, and $\ss_1(s_{0,n}) = \ss_2(s_{m,0}) = 0$ for all $m, n \ge 1$.
\end{enumerate}
%
%
\noindent
{\bf Proof of \ref{item:properties_of_smn(a)}.} All the operators that appear in Eq.~\eqref{equ:operator_Tmn} commute with each other. Accordingly, we compute that
$\dd_x(1) = \dd_y(1) = 0$, while $(\dd_x)_x' = \LL_x +D_x(\LL_x)_x'$ and $(\dd_y)_y' = \LL_y +D_y(\LL_y)_y'$. It follows that
$$
s_{0,0}(x,y) = \LL_x^{-1}  \LL_x \LL_y^{-1} \LL_y(1) = 1.
$$
{\bf Proof of \ref{item:properties_of_smn(c)}.}
Notice that $\ss_1 = D_x\LL_x\EE_{(a_{1,1},a_{2,1})}$  commutes with all the operators involved in
the definition of $s_{m,n}$ (but not with $x^m$ or $y^n$). Therefore, we find that, for all $m, n \in \mathbf N$ such that $m+n \ge 1$, $\mathfrak s_1(s_{m,n}(x,y))$ is equal to
\begin{equation*}
\left|
\begin{array}{cc}
(\dd_x)_x' +a_{1,1}\dd_x  &  a_{2,1}\dd_x\\
a_{1,2}\dd_y  &  (\dd_y)_y' +a_{2,2}\dd_y
\end{array}
\right| \nonumber
\EE_{(a_{1,1}, a_{2,1})}^{-m+1}\EE_{a_{1,2},a_{2,2}}^{-n}
\LL_x^{-m} \LL_y^{-n} (mx^{m-1} y^n),
\end{equation*}
that is, $\ss_1 (s_{m,n}(x,y)) = m T_{m-1,n}(x^m y^n) = ms_{m-1,n}(x,y)$.
In a similar way, we see that $\ss_2(s_{m,n}) = ns_{m,n-1}$ for all $m,n \in \mathbb N$ with $m+n \ge 1$.
\vskip 0.1cm
\noindent {\bf Proof of \ref{item:properties_of_smn(b)}.} For all $m, n \geq 0$, we have from the definition of $s_{m,n}$
and the commutativity of shift-invariant operators that
\begin{equation} \label{proof-of-b}
s_{m,n}(x,y) =
 \EE_{(a_{1,1}, a_{2,1})}^{-m}\EE_{(a_{1,2},a_{2,2})}^{-n}(\mathcal H_{m,n}(x,y))
\end{equation}
where the polynomial $\mathcal H_{m,n}(x,y)$ is given by the $2 \times 2$ determinant
\begin{equation*}
\left|
\begin{array}{cc}
\!(\dd_x)_x'\LL_x^{-m-1}(x^m) +a_{1,1}D_x \LL_x^{-m}(x^m)  &  a_{2,1}\LL_x^{-m}D_x (x^m)\!  \\

\! a_{1,2}\LL_y^{-n}D_y (y^n)   &  (\dd_y)_y'\LL_y^{-n-1}(y^n) + a_{2,2}D_y\LL_y^{-n}(y^n)\!
\end{array}
\right|\!.
\end{equation*}
Now we combine the previous equations with the univariate formulas  \eqref{pincherle equation} and  \eqref{power equation}, which we use in the form
$$
(\dd_x)'_x \LL^{-m-1}(x^m) = x\LL^{-m}(x^{m-1}).
$$
If $m \geq 1$ and $n \geq 1$, we get
\small
\begin{equation*}
\mathcal H_{m,n}(x,y) = \left|
\begin{array}{cc}
\!x \LL_x^{-m}(x^{m-1}) +ma_{1,1} \LL_x^{-m}(x^{m-1}) &  ma_{2,1}\LL_x^{-m} (x^{m-1})  \!\\
\!na_{1,2}\LL_y^{-n} (y^{n-1})   &  y\LL_y^{-n}(y^{n-1}) + na_{2,2}\LL_y^{-n}(y^{n-1}) \!
\end{array}
\right|\!,
\end{equation*}
\normalsize
which ultimately leads to
\begin{equation*}
s_{m,n}(x,y) = \EE_{(a_{1,1}, a_{2,1})}^{-m}\EE_{(a_{1,2},a_{2,2})}^{-n}
\left|
\begin{array}{cc}
x +ma_{1,1}  &  ma_{2,1}  \\
na_{1,2}   &  y+ na_{2,2}
\end{array}
\right|
\LL_x^{-m} \LL_y^{-n}
(x^{m-1} y^{n-1}).
\end{equation*}
In these computations, the monomial $x^m$ (respectively, $y^n$) must stay
to the right of any operator acting on $x^m$ (respectively, on $y^n$).

Applying the shifts, and taking into account that $\EE^{-m}_{(a_{1,1},a_{2,1})} = \EE_{(-ma_{1,1},-ma_{2,1})}$ and $\EE_{(a_{1,2},a_{2,2})}^{-n} = \EE_{(-na_{1,2},-na_{2,2})}$, we obtain that $s_{m,n}(x,y)$ equals
\small
\begin{equation*}
\left|
\begin{array}{cc}
x -na_{1,2}  &  ma_{2,1}  \\
na_{1,2}   &  y - ma_{2,1}
\end{array}
\right|
\LL_x^{-m} \LL_y^{-n}(
(x-ma_{1,1}-na_{1,2})^{m-1} (y-ma_{1,1} -na_{2,2})^{n-1}),
\end{equation*}
\normalsize
which simplifies to give
\begin{equation}
\label{bivariate after shifts}
s_{m,n}(x,y) = \left|
\begin{array}{cc}
x -x_{0,n} &  y_{m,0}  \\
x_{0,n}   &  y - y_{m,0}
\end{array}
\right|
\LL_x^{-m} \LL_y^{-n}((x-x_{m,n})^{m-1}   (y-y_{m,n})^{n-1}).
\end{equation}
Evaluating at $(x,y) = (0,0)$,  we get $s_{m,n}(0,0) = 0$.

On the other hand, if one of $m$ and $n$ is zero, say $m \geq 1$ and $n=0$,
then, starting from Eq.~\eqref{proof-of-b},
using \eqref{pincherle equation} and  \eqref{power equation} only on $x$, and recalling that $(\dd_y)'_y= \LL_y+D_y(\LL_y)'_y$, we have
\small
\begin{equation} \label{n=0}
\begin{split}
s_{m,0}(x,y) & =
\EE_{(a_{1,1}, a_{2,1})}^{-m}
\left|
\begin{array}{cc}
\!x\LL_x^{-m}(x^{m-1}) +a_{1,1} m  \LL_x^{-m}(x^{m-1})  &  a_{2,1}m \LL_x^{-m} (x^{m-1})\!  \\
0   &   1
\end{array}
\right|  \\
&=   \EE_{(a_{1,1},a_{2,1})}^{-m} [x+ma_{1,1}]\LL_x^{-m}(x^{m-1}) \\
&=  x\LL_x^{-m}  ((x-ma_{1,1})^{m-1}).
\end{split}
\end{equation}
\normalsize
Evaluating at $(0,0)$ yields $s_{m,0}(0,0) = 0$. Similarly for $m=0$ and $n \geq 1$.  This confirms \ref{item:properties_of_smn(b)}.

\noindent \textbf{Step 3.} We prove \ref{nodes}.
We  have from Eq.~\eqref{power equation} that
$$
\LL_x^{-m}(x^{m-1}) = \frac{p_m(x)}{x}\ \ \text{and}\ \ \LL_y^{-n}(y^{n-1}) = \frac{q_n(y)}{y}.
$$
These, together with Eqs.~\eqref{bivariate after shifts} and \eqref{n=0},
yields \ref{nodes}.
\end{proof}
\noindent
By way of example, we now consider some special cases of the general formulas implied by Theorem \ref{bivariate Abel}.
%
\begin{example}
If $A$ is the identity matrix, then the basic sequence of the pair $(\dd_x, \dd_y)$ is $\{p_m(x)q_n(y)\}_{(m,n) \in \mathbb{N}^2}$, and
the delta Gon\v{c}arov polynomials are
$$
t_{m,n}((x,y);Z)=\frac{x p_m(x-m)}{x-m} \cdot \frac{y q_n(y-n)}{y-n} .
$$
\end{example}
\begin{example}
For the classical Abel polynomials, where $\dd_x=D_x$ and $\dd_y=D_y$, we have
$p_m(x)=x^m$ and $q_n(y)=y^n$. Hence Theorem \ref{bivariate Abel}\ref{nodes} provides the same formula \eqref{classical} obtained in \cite{[Lorentz-yan]}.
\end{example}
%
\begin{example}
Forward differences are given by the delta operator $\dd = \EE_1 -\mathfrak{I} = D\LL$, where
$\LL f(x) = \int_x^{x+1}f(t)dt$ is the Bernoulli operator. The $n$-th basic pol\-y\-no\-mi\-al for this operator
is the lower factorial function
$$
p_n(x) = (x)_{(n)}=x(x-1) \cdots (x-n+1).
$$
So, if the delta operators $\dd_x$ and $\dd_y$ are forward difference operators relative to $x$ and $y$ respectively,  then the delta Abel polynomial $t_{m,n}((x,y);Z)$ is equal to
$$
\left[ (x-x_{0,n})(y-y_{m,0}) - x_{0,n}y_{m,0}\right]
\cdot (x-x_{m,n}-1)_{(m-1)}(y-y_{m,n}-1)_{(n-1)}.
$$
%
%
Similarly, backward differences are given by the delta operator $\dd = \mathfrak{I}-\EE_{-1} = D\LL$,
with $\LL f(x) =\int_{x-1}^x f(t) dt$. The $n$-th basic polynomial associated to this operator is the upper factorial function
$$
p_n(x) = (x)^{(n)}=x(x+1) \cdots (x+n-1).
$$
So, if the delta operators $\dd_x$ and $\dd_y$ are, respectively, backward difference operators relative to $x$ and $y$,
then the delta Abel polynomial $t_{m,n}((x,y);Z)$ equals
$$
\left[ (x-x_{0,n})(y-y_{m,0}) - x_{0,n}y_{m,0}\right]
\cdot (x-x_{m,n}+1)^{(m-1)}(y-y_{m,n}+1)^{(n-1)}.
$$
\end{example}
%
%
%
\section{Closed formulas for multivariate delta Abel polynomials}
In this section we give closed formulas for those multivariate delta Abel polynomials whose system of delta operators $(\dd_1, \dd_2, \dots, \dd_d)$ is separable, which means that, for each $i = 1,\ldots, d$,
\begin{equation}
\dd_{i}
= D_{i} \sum_{j=0}^\infty a_j^{(i)}D_{i}^j
= D_{i}  \LL_{i},
\label{delta d-tuple}
\end{equation}
for some real coefficients $a_0^{(i)}, a_1^{(i)}, \ldots$ with $a_0^{(i)} \neq 0$.

Let $\{p_n^{(i)} \}_{n \in \mathbb N}$ be the  basic sequence
of $\dd_i$. Note that $p_n^{(i)}$ is only a function of the variable $x_i$
and that the delta operator $\dd_{i}$
contains only partial derivatives with respect to $x_i$.
The grid $Z$ is a linear transformation of $\mathbb{N}^d$ by a certain
$d\times d$ matrix $A=(a_{i,j})$.
%
%
%
\begin{theorem} \label{general-closed}
Let the system of delta operators $(\dd_1, \dd_2, \dots, \dd_d)$, the matrix  $A$,  and  the grid
$Z$ be defined as above.
Let $B$ be the $d \times d$ diagonal matrix with
\begin{equation} \label{matrix-B}
B_{i,i} = x_i - z_{\mathbf{n},i}
 \qquad (1\leq i \leq d),
\end{equation}
and $C$ the $d \times d$ matrix whose $(i,j)$-entry is given by
\begin{equation} \label{matrix-C}
C_{i,j} = z_{n_i \mathbf{e_i}, j},
\end{equation}
where, for a vector $\mathbf{k} \in \mathbb{N}^d$, $z_{\mathbf{k}, j}$ is the $j$-th entry of the point $z_{\mathbf{k}}$.
Moreover, let $F$ be the $d\times d$ diagonal matrix of operators with
\begin{equation} \label{matrix-F}
F_{i,i} = (\dd_i)'_{x_i}
\qquad (1 \leq i \leq d),
\end{equation}\label{matrix-G}
and $G$ the $d \times d$ matrix of operators whose $(i,j)$-entry is $G_{i,j} = a_{i,j} \dd_i$.

Denote by $(t_{\mathbf{n}}(\bbx; Z))_{\mathbf n \in \mathbb N^d}$ the sequence of delta Abel polynomials
 associated with $(\dd_1, \dd_2, \dots, \dd_d)$ and the grid $Z$, and let $\mathbf n = (n_1, n_2, \ldots, n_d) \in \mathbf N^d$ be fixed. The following statements hold:
\begin{enumerate}[label={\rm(\roman{*})}]
\item We have
\begin{equation} \label{general-close1}
t_{\mathbf{n}}(\mathbf{x}; Z) =
\det(B+C)
\prod_{i=1}^d
\frac{ p^{(i)}_{n_i}(x_i-z_{\mathbf{n},i}) }
{x_i-z_{\mathbf{n},i}}.  
\end{equation}

\item Let $\ss_i= \dd_i \EE_{\mathbf{t}_i}$, where $\mathbf{t}_i$ is the $i$-th column vector of the matrix $A$. Then
\begin{equation} \label{general-close2}
t_{\mathbf{n}}(\mathbf{x}; Z) =
J((\mathbf{\ss})_{\mathbf{x}}')
\left(\prod_{i=1}^d \LL_{i}^{-n_i-1} \right)
\left( \prod_{i=1}^d \EE_{\mathbf{t}_i}^{-n_i-1} \right)  \mathbf{x}^{\mathbf{n}},
\end{equation}
where $J((\mathbf{\ss})_{\mathbf{x}}')$ is the Jacobian matrix whose $(i,j)$-entry is $(\ss_i)'_{x_j}$.

\item We have
\begin{equation}\label{general-close3}
t_{\mathbf{n}}(\mathbf{x}; Z) =
\det(F+G)
\left(\prod_{i=1}^d \LL_{i}^{-n_i-1} \right)
\left(\prod_{i=1}^d \EE_{\mathbf{t}_i}^{-n_i} \right)
 \mathbf{x}^{\mathbf{n}}.
\end{equation}
\end{enumerate}
\end{theorem}
\begin{proof}
 We only sketch the proof here since the computations are all similar to those in the bivariate case.
\vskip 0.1cm
\noindent \textbf{Step 1}. We prove that \eqref{general-close2} and \eqref{general-close3} are equivalent
by  calculating  that
\begin{eqnarray*}
 (\ss_i)'_{x_j} = \left\{
  \begin{array}{ll}
    (\dd_i)'_{x_i} \EE_{\mathbf{t}_i} + a_{i,i} \dd_i \EE_{\mathbf{t}_i} & \text{if } i=j \\
    a_{j,i} \dd_i \EE_{\mathbf{t}_i}  & \text{if } i\neq j
  \end{array} \right.\!.
\end{eqnarray*}
Substituting into \eqref{general-close2} we obtain \eqref{general-close3}.
\vskip 0.1cm
\noindent \textbf{Step 2}. We check that Eq. \eqref{general-close3} yields the basic sequence of
the system $(\ss_1, \ss_2, \dots, \ss_d)$, since then this is also the desired
delta Gon\v{c}arov polynomials, by Proposition \ref{shift-prop}.
For that, denote by $s_{\mathbf{n}}(\mathbf{x})$ the right-hand side of
\eqref{general-close3}. By induction, it suffices to prove that:
\begin{enumerate}[label={\rm(\alph{*})}]
\item\label{item:final_(a)}
    $s_{\mathbf{0}}(\bbx) = 1$;
\item\label{item:final_(b)}
    $s_{\mathbf{n}}(0,0) = 0$ if $\mathbf{n}\neq \mathbf{0}$;
\item\label{item:final_(c)}
    $\ss_i(s_{\mathbf{n}}) = n_i s_{\mathbf{n}-\mathbf{e}_i}$ for all $i = 1, \ldots, d$.
\end{enumerate}
Point \ref{item:final_(a)} is immediate, when considering that $\dd_i(1) =D_i(1)=0$ and
$$
(\dd_i)'_{x_i} \LL_i^{-1}(1) = (\LL_i +D_i (\LL_i)'_{x_i} )\LL_i^{-1}(1)= 1.
$$
As for \ref{item:final_(b)}, we use again Equations ~\eqref{pincherle equation} and \eqref{power equation}
to get $(\dd_i)'_{x_i} \LL_i^{-n_i-1} x_i^{n_i} = x_i \LL_i^{-n_i} x^{n_i-1}$. Hence if $n_i \geq 1 $ for all $i$,
$s_{\mathbf{n}}(\mathbf{x})$ can be expressed as
 \begin{equation*}
 s_{\mathbf{n}}(\mathbf{x}; Z)=
 \left(  \prod_{i=1}^d \EE_{\mathbf{t}_i}^{-n_i} \right)
  \det(K(\bbx)) \prod_{i=1}^d \LL_{i}^{-n_i} (x_i^{n_i-1}),
  \end{equation*}
where $K(\bbx)$ is the $d \times d$ matrix with entries
\begin{equation*}
 k_{i,j} =\left\{
 \begin{array}{ll}
   x_i + n_i a_{i,i} & \text{if } i=j \\
   n_i a_{j,i}   & \text{if }  i \neq j
 \end{array} \right.\!.
\end{equation*}
Applying the shift operators first, we have
\begin{eqnarray} \label{Matrix M}
 s_{\mathbf{n}}(\mathbf{x})= \det(M(\bbx)) \prod_{i=1}^d \left(\LL_{i}^{-n_i} ((x_i-z_{\mathbf{n},i})^{n_i-1}) \right) ,
\end{eqnarray}
where $M(\bbx)$ is the $d \times d$ matrix with entries
\begin{equation*}
 m_{i,j} =\left\{
 \begin{array}{ll}
  x_i + n_i a_{i,i}-\sum_{j=1}^d n_j a_{i,j} & \text{if } i=j \\
   n_i a_{j,i}   & \text{if }  i \neq j
 \end{array} \right.\!.
\end{equation*}
When $\bbx=\mathbf{0}$, $\det (M)=0$ since every column of $M(\mathbf{0})$ has sum zero. Hence $s_\mathbf{n}(\mathbf{0})=0$ whenever
$\mathbf{n}$ has no zero entry.

For the case that some $n_i=0$ but $\mathbf{n} \neq \mathbf{0}$, the computation  reduces to an equivalent one at a lower dimension
and with an index vector whose entries are all positive.
Since the above argument applies to all dimensions $d \geq 1$, we conclude
$s_\mathbf{n}(\mathbf{0})=0$  for all $\mathbf{n} \neq \mathbf{0}$.

As for \ref{item:final_(c)}, note that  $\ss_i=D_i \LL_i \EE_{\mathbf{t}_i}$ commutes with all operators in the definition of $s_{\mathbf{n}}$.
Hence we can compute  $\ss_i(s_{\mathbf{n}})$ by applying $\ss_1$ to $\mathbf{x}^\mathbf{n}$ first, then applying the other operators
in the formula of $s_{\mathbf{n}}$. This yields exactly the formula of $n_i s_{\mathbf{n}-\mathbf{e_i}}$.
\vskip 0.1cm
\noindent \textbf{Step 3}.  We derive Eq.~\eqref{general-close1} from \eqref{general-close3}.
Notice that in terms of $\mathbf{n}$ and the matrix $A=(a_{i,j})$, Eqs.~\eqref{matrix-B} and \eqref{matrix-C} can be
re-written as
$$
 B_{i,i}= x_i - \sum_{j=1}^d n_j a_{i, j}, \qquad
 C_{i,j} = n_i a_{j,i}.
 $$
Hence the matrix $M$  in \eqref{Matrix M} is exactly $B + C$.
Formula \eqref{general-close1} is then obtained by using the equation
$$
 \LL_i^{-n_i}((x-z_{\mathbf{n},i})^{n_i-1}) = \frac{ p^{(i)}_{n_i}(x_i-z_{\mathbf{n},i})}{x_i-z_{\mathbf{n},i}}.
$$
This completes our proof.
\end{proof}

\begin{remark}
Given a grid $Z \subseteq \mathbb R^d$ and a vector $\mathbf{v} \in \mathbb R^d$, let
$
Z+ \mathbf{v} =\{ z_\mathbf{k} + \mathbf{v} : z_\mathbf{k} \in Z\}$.
It is easy to verify that the delta Gon\v{c}arov polynomials are translation invariant, i.e., they satisfy
$t_\mathbf{n}(\bbx +\mathbf{v}; Z+\mathbf{v})= t_\mathbf{n}(\bbx; Z)$.

As a consequence, if $Z'$ is an affine transformation of $\mathbb{N}^d$, i.e., $Z'=\{ z'_{\mathbf{k}} = \mathbf{v} + A \mathbf{k}:
 \mathbf{k} \in \mathbb{N}^d\}$ for some $d \times d$ matrix $A$, then Theorem \ref{general-closed} yields a closed formula for the delta
 Gon\v{c}arov polynomials associated with the grid $Z'$ by the equation
 $  t_\mathbf{n}(\bbx; Z') = t_{\mathbf{n}}(\bbx-\mathbf{v}; Z)$,
where $Z=\{z_\mathbf{k}: \mathbf{k} \in \mathbb{N}^d, z_{\mathbf{k}}=A\mathbf{k} \}$.
\end{remark}
\begin{remark}
In \cite[Theorem (4.3)]{[Lorentz-yan]} a formula was given for the classical trivariate Abel polynomials.
 It was noted that the formula there was not valid for all linear matrices and suggested the  question
of  finding a  closed general formula for the classical Abel polynomials in three variables.

Since in the classical case, the system of delta operators is separable,  Eq.~\eqref{general-close1}
gives a complete answer for this question.
In dimension 3  let the variables be $x, y, z$, then \eqref{general-close1}  specializes to
\small
\begin{equation}\label{trivariate Abel with nodes}
\begin{split}
t_{m,n,p}((x,y,z); Z) & = \left|
\begin{array}{ccc}
x -x_{0,n,p} & y_{m,0,0} & z_{m,0,0}\\
x_{0,n,0} & y-y_{m,0,p} & z_{0,n,0} \\
x_{0,0,p} & y_{0,0,p} & z -z_{m,n,0}
\end{array}
\right| \cdot \\[0.8ex]
& \cdot \frac{ p_{m}(x-x_{m,n,p})}{x-x_{m,n,p}} \cdot
\frac{q_{n}(y-y_{m,n,p})}{y-y_{m,n,p}} \cdot
\frac{r_{p}(z-z_{m,n,p})}{z-z_{m,n,p}},
\end{split}
\end{equation}
\normalsize
where $\{ p_m\}$, $\{ q_n\}$ and $\{ r_p\}$ are the basic sequences of the univariate
operators $\dd_x$, $\dd_y$ and  $\dd_z$ respectively.
 When  $(\dd_x, \dd_y, \dd_z)=(D_x, D_y, D_z)$,
using  $p_m(x)=x^m$, $q_n(y)=y^n$ and $r_p(z)=z^p$, we obtain from  \eqref{trivariate Abel with nodes}
a complete formula for the classical trivariate Abel polynomials.
\end{remark}
\section*{Acknowledgments}
We are grateful to an anonymous referee  for reading the manuscript carefully and helping us improve the presentation of the paper.
This publication was made possible by NPRP grant No. [5-101-1-025] from the Qatar National
Research Fund (a member of Qatar Foundation). The statements made herein are
solely the responsibility of the authors.
%
%

\end{document}